\documentclass[12pt]{amsart}

\usepackage[greek,english]{babel}
\usepackage[iso-8859-7]{inputenc}
\usepackage{graphicx}
\usepackage{hyperref}
\usepackage{color}
\usepackage{amsthm, amsfonts, amsmath, amstext}
\usepackage{amssymb}
\usepackage{url}
\usepackage[shortlabels]{enumitem}
\newcommand{\ol}[1]{\overline{#1}}
\newcommand{\mc}[1]{\mathcal{#1}}

\newcommand{\ints}{\mathbb{Z}}
\newcommand{\proj}{\mathbb{P}}

\newtheorem{theorem}{Theorem}[section]
\newtheorem{proposition}[theorem]{Proposition}
\newtheorem{lemma}[theorem]{Lemma}
\newtheorem{corollary}[theorem]{Corollary}
\newtheorem{conjecture}[theorem]{Conjecture}
\newtheorem{definition}[theorem]{Definition}
\newtheorem{problem}[theorem]{Problem}

\theoremstyle{remark}
\newtheorem{example}[theorem]{Example}
\newtheorem{remark}[theorem]{Remark}

\begin{document}
\begin{abstract}
It is conjectured that if $k$ is an algebraically closed field of characteristic $p > 0$, then any branched $G$-cover of smooth projective $k$-curves where the ``KGB'' obstruction vanishes and where a $p$-Sylow subgroup of $G$ is cyclic lifts to characteristic $0$. Obus has shown that this conjecture holds given the existence of certain meromorphic differential forms on $\proj^1_k$ with behavior determined by the ramification data of the cover. We give a more efficient computational procedure to compute these forms than was previously known. As a consequence, we show that all $D_{25}$- and $D_{27}$-covers lift to characteristic zero.   
\end{abstract}

\title{Local Oort groups and the isolated differential data criterion}
\author[Dang]{Huy Dang}
\address{University of Virginia}
\curraddr{141 Cabell Drive, Charlottesville, VA 22903}
\email{hqd4bz@virginia.edu}

\author[Das]{Soumyadip Das}
\address{Indian Statistical Institute, Bangalore Centre}
\curraddr{8th Mile, Mysore Road, Bangalore 560059, India}
\email{soumyadip\_rs@isibang.ac.in}

\author[Karagiannis]{Kostas Karagiannis}
\address{Aristotle University of Thessaloniki}
\curraddr{Department of Mathematics,			
School of Sciences,
54124, Thessaloniki, Greece}
\email{kkaragia@math.auth.gr}

\author[Obus]{Andrew Obus}
\address{Baruch College}
\curraddr{1 Bernard Baruch Way. New York, NY 10010, USA}
\email{andrewobus@gmail.com}

\author[Thatte]{Vaidehee Thatte}
\address{Binghamton University}
\curraddr{Binghamton, New York 13902-6000, USA}
\email{thatte@math.binghamton.edu}

\maketitle

\tableofcontents
\section{Introduction}

This paper concerns the \emph{local lifting problem}, which is stated as follows:

\begin{problem}[The local lifting problem]\label{Plocallifting}
Let $k$ be an algebraically closed field of characteristic $p$ and $G$ a finite group. Let $k[[z]]/k[[t]]$ be a $G$-Galois extension (that is, $G$ acts on $k[[z]]$ by $k$-automorphisms with fixed ring $k[[t]]$). Does this extension lift to characteristic zero?  That is, does there exist a DVR $R$ of characteristic zero with residue field $k$ and a $G$-Galois extension $R[[Z]]/R[[S]]$ that reduces to $k[[z]]/k[[s]]$?  
\end{problem}

Let us give some brief context --- for more details, see the expositions \cite{Ob:ll} and \cite{Ob:lc}.  The local lifting problem is motivated by the \emph{global lifting problem}, which asks whether a characteristic $p$ curve with a finite group of automorphisms (or, equivalently, a Galois branched cover of curves) lifts to characteristic zero.  In fact, solving the global lifting problem is equivalent to solving the local lifting problem for each extension coming from the complete local ring of a ramification point on the cover.  For tame covers, this reduces to the local lifting problem when $G$ is cyclic and prime to $p$, which is more or less trivial by Kummer theory, and gives an alternate proof of one of the main results of SGA1 (\cite[XIII, Corollaire 2.1]{SGA1}).  For more on this \emph{local-global principle}, see \cite[\S3]{Ob:lc}, or see the papers \cite{Ga:pr}, \cite{GM:lg}, and \cite{BM:df} for the original proofs.

We will refer to a $G$-Galois extension $k[[z]]/k[[t]]$ for $k$ an algebraically closed field of characteristic $p$ as a
\emph{local $G$-extension}.
Basic ramification theory shows that any group $G$ that occurs as the Galois group of a local extension is of the form $P \rtimes \ints/m$, with $P$ a $p$-group and $p \nmid m$.
In \cite{CGH:ll}, Chinburg, Guralnick, and Harbater
ask, given a prime $p$, for which groups $G$ (of the form $P \rtimes \ints/m$) is it true that
all local $G$-actions (over all algebraically closed fields of characteristic $p$) lift to
characteristic zero? Such a group is called a \emph{local Oort group}
(for $p$).  Due to various obstructions (The \emph{Bertin obstruction} of \cite{Be:ol}, the \emph{KGB obstruction} of \cite{CGH:ll}, and the \emph{Hurwitz tree obstruction} of \cite{BW:ac}), the list of possible local Oort groups is quite limited. In particular, the following proposition is a consequence of  \cite[Theorem 1.2]{CGH:ll} and \cite{BW:ac}.

\begin{proposition}\label{Poortlist}
If a group $G$ is a local Oort group for $p$, then $G$ is either cyclic, dihedral of order $2p^n$, or the alternating group $A_4$ (with $p=2$). 
\end{proposition} 

Cyclic groups are known to be local Oort --- this is the so-called \emph{Oort conjecture}, proven by Obus--Wewers and Pop in
\cite{OW:ce}, \cite{Po:oc}.  Obus proved that $A_4$ is local Oort in \cite{Ob:A4} (this was also independently known to Pop and Bouw).  This leaves the case of dihedral groups.  

Dihedral groups of order $2p$ are known to be local Oort for $p$ odd due to Bouw--Wewers (\cite{BW:ll}) and for $p = 2$ due to Pagot (\cite{Pa:rc}).
The group $D_9$ is local Oort by \cite{GenOort}, and the group $D_4$ is local Oort by \cite{We:D4}.  No other dihedral groups are known to be local Oort.  Our main theorem is:

\begin{theorem}
The groups $D_{25}$ and $D_{27}$ are local Oort.
\end{theorem}

In fact, the paper \cite{GenOort} states a more general conjecture for groups with cyclic $p$-Sylow subgroup.  The conjecture below is a combination of Conjecture 1.9, Proposition 1.6, and Remark 1.7 of \cite{GenOort}.

\begin{conjecture}\label{Cmain}
Let $G \cong \ints/p^n \rtimes \ints/m$ be non-abelian with $p \nmid m$, and let $k[[z]]/k[[t]]$ be a local $G$-extension whose $\ints/p^n$-subextension has ramification jumps $(u_1, \ldots, u_n)$ for the upper numbering (\cite[IV]{Se:lf}) that are congruent to $-1 \pmod m$.  Then $k[[z]]/k[[t]]$ lifts to characteristic zero.
\end{conjecture}

\begin{remark}\label{Rcenterfree}
By \cite[Remark 1.7]{GenOort} (see also \cite[Remark 1.10]{OP:wt}), for a local $G$-extension to satisfy the hypotheses of Conjecture~\ref{Cmain}, it is necessary that $G$ be center-free.  That is, the conjugation action of $\ints/m$ on $\ints/p^n$ is faithful. 
\end{remark}

\begin{remark}
If the congruence condition of Conjecture~\ref{Cmain} is \emph{not} satisfied, then it is known that the local $G$-extension does \emph{not} lift to characteristic zero (see \cite[Proposition 1.6]{GenOort}).  
\end{remark}

\begin{remark}
If $m = 2$ (so that $G$ is dihedral), then the $u_i$ are \emph{always} odd, so Conjecture~\ref{Cmain} asserts that $D_{p^n}$ is local Oort for odd $p$ as a special case.
\end{remark}

In \cite{GenOort}, Conjecture~\ref{Cmain} was reduced to showing that the so-called \emph{isolated differential data criterion} holds in sufficiently many cases.  Saving the details for \S\ref{Sdiff}, we say now that this criterion is about finding a meromorphic differential form on $\proj^1_k$ with prescribed poles that transforms in a particular way under the Cartier operator, and which is ``isolated" in the sense that small deformations of the differential form do not satisfy these criteria.  In fact, constructing the differential form can be reduced to constructing a certain polynomial $f \in k[x]$, and the isolatedness criterion is stated in terms of the invertibility of a Vandermonde-like determinant arising from the roots of $f$.  One of the key intermediate results of this paper is Corollary~\ref{Cisolated1}, which rewrites this criterion in terms of the \emph{coefficients} of $f$.  This allows us to write an algorithm to verify the \emph{existence} of a satisfactory $f$ entirely in terms of Gr\"{o}bner bases (even if we cannot write the solution explicitly).  In particular, this existence criterion is necessary for us to prove that $D_{27}$ is local Oort.

\subsection{Outline of the paper}\label{Soutline}
In \S\ref{Sdiff} we discuss the isolated differential data criterion and its relation to the local lifting problem for dihedral groups, introduced in \cite{GenOort}. The criterion is defined in two steps: in Definition \ref{def_ddc} we state the differential data criterion, which is about the existence of a meromorphic differential form on $\mathbb{P}_k^1$ with pre-specified behavior under the Cartier operator, while the notion of isolatedness is made specific in Definition \ref{def_isolated}. The main result of this section is Proposition \ref{prop_best_bound} which essentially reduces the local lifting problem to verifying that the isolated differential data criterion holds for finitely many cases.

The meromorphic differential form of Definition \ref{def_ddc} is uniquely determined by a polynomial $f\in k[t]$. In the two subsequent sections, we reformulate each of the two conditions that make up the isolated differential data criterion so as to be expressed in terms of the vanishing or non-vanishing of certain polynomials in the coefficients of $f$. In particular, \S\ref{Smultinomial} deals with the differential data criterion - ignoring the isolatedness conditions. Its main result is Proposition \ref{system}, which says that the differential data criterion is equivalent to a system of equations in the coefficients of $f$, which involve multinomials obtained by the expansion of the polynomial $f^{p-1}$.

In \S\ref{Sisolated} we discuss the isolatedness condition and translate it to a condition in terms of the coefficients of $f$ rather than its roots.  Our result makes use of Heineman's work on generalized Vandermonde matrices and their determinants. The main result of this section is Corollary \ref{Cisolated1} in which we prove that the isolatedness condition is equivalent to the invertibility of a matrix whose entries are uniquely determined by the coefficients of $f$.

The main results of the two previous sections are combined in \S\ref{Sgrobner} in two different manners: Remark \ref{algorithm-1} describes an approach which requires first solving the system of Proposition \ref{system} then checking whether the respective matrix defined in Corollary \ref{Cisolated1} is invertible. The difficulty of explicitly solving the system of Proposition~\ref{system} motivates the existence criterion of Proposition \ref{grobner-criterion}, the main result of this section, in which we prove that the isolated differential data criterion holds if and only if an ideal uniquely determined by the equations of Proposition \ref{system} and Corollary \ref{Cisolated1} is not the unit ideal.

Finally, in \S\ref{Sresults} we use our results to prove that $D_{25}$ and $D_{27}$ are local Oort groups.  In fact, the small size of the input data allows us to explicitly realize the isolated differential data criterion as in Remark~\ref{algorithm-1} for all $D_{25}$ cases and all but two $D_{27}$ cases.  In these two cases, we must rely on the existence criterion of Proposition \ref{grobner-criterion}. 


\section*{Acknowledgements}
This project was conceived at the AMS Mathematics Research Community ``Explicit methods in Arithmetic Geometry in Characteristic $p$" in June 2019, and we thank the AMS and the organizers of that workshop.  In particular, we thank Drew Sutherland and Sachi Hashimoto for fruitful conversations during the workshop.  The second author would like to thank Indian Statistical Institute for partial travel support for the aforementioned workshop.

This material is based upon work supported by the National Science Foundation under Grant Numbers DMS-1641020 and DMS-1900396.  Support for this project was also provided by a PSC-CUNY Award, jointly funded by The Professional Staff Congress and The City University of New York.  

\section{The isolated differential data criterion}\label{Sdiff}
In this section we briefly recall the notion of the isolated differential data criterion following \cite[\S1.4]{GenOort}.
\begin{definition}\label{def_ddc}
Let $p$ be a prime number and $k$ be an algebraically closed field of characteristic $p$. Let $m > 1$ be an integer dividing $p-1$, and $\tilde{u}$ be a positive integer such that $\tilde{u} \equiv -1 \pmod{m}$. Let $N \in \{\tilde{u}(p-1), \tilde{u}(p-1) - m\}$. Define $u$ by $\tilde{u} = u p^\nu$, $p \nmid u$. We say that the differential data criterion is satisfied for the quadruple $(p,m,\tilde{u},N)$ (with respect to the field $k$) if there exists a polynomial $f(t) \in k[t^m]$ of degree $N$ such that the meromorphic differential form $\omega := \frac{dt}{f(t) t^{\tilde{u}+1}} \in \Omega^1_{k(t)/k}$ satisfies
\begin{equation}\label{d.d.c.}
\mathcal{C}(\omega) = \omega + u t^{-\tilde{u}-1} dt,
\end{equation}
where $\mathcal{C}$ is the Cartier operator on $\Omega^1_{k(t)/k}$.
\end{definition}

If the differential data criterion is realized by a meromorphic differential form $\omega$ (or equivalently, for an element $f(t) \in k[t^m]$), we will say that $\omega$ (or $f(t)$) is a \textit{solution} to the differential data criterion for $(p,m,\tilde{u},N)$.  

\begin{lemma}\label{Lfroots}
If $f(t)$ is a solution to the differential data criterion for $(p, m, \tilde{u}, N)$, then $f(t)$ is separable and is not divisible by $t$.
\end{lemma}

\begin{proof}
If $\alpha$ is a root of $f(t)$, then by (\ref{d.d.c.}), the order of the pole of $\mc{C}(\omega)$ at $t = \alpha$ is the same as that of $\omega$.  From the basic properties of the Cartier operator (see \cite{MR84497}), the order of this pole is $1$.  Since $\tilde{u} \geq 1$, this is a contradiction for $\alpha = 0$.  If $\alpha \neq 0$, this shows that $\alpha$ has multiplicity one as a root of $f(t)$.
\end{proof}

From the basic properties of the Cartier operator $\mathcal{C}$ it follows that if $\omega$ is a solution to the differential data criterion for the quadruple $(p,m,\tilde{u},N)$, then $\omega$ must be of the form
\begin{equation}\label{eq_omega_form}
\omega = dg/g - u \sum_{i=0}^\nu t^{-up^i -1} dt
\end{equation}
for some rational function $g \in k(t)$.  Since $g$ appears only in the term $dg/g$, it is well defined up to multiplication by $p^{\text{th}}$ powers. Also note that $\omega$  has a zero of order $N+\tilde{u}-1$ at $t=\infty$.



\begin{definition}\label{def_isolated}
Let $p$, $m$, $\tilde{u}$, $N$ be as in Definition \ref{def_ddc}. Let $\omega$ be a solution to the differential data criterion for $(p,m,\tilde{u},N)$ where $\omega = dg/g - u \sum_{i=0}^\nu t^{-up^i -1} dt$ (cf.\ (\ref{eq_omega_form})). We say that the isolated differential data criterion holds for $(p,m,\tilde{u},N)$ if no infinitesimal deformation $\tilde{g}$ of $g$ gives rise to a differential form $\widetilde{\omega} = d \tilde{g}/ \tilde{g} - u \sum_{i=0}^\nu t^{-up^i -1} dt$ having a zero of order at least $N+\tilde{u}-1$ at $t =\infty$.
\end{definition}

Let $f(t)$ be a solution to the differential data criterion for $(p,m,\tilde{u},N)$ and write $x_1, \ldots, x_{N/m}$ for a set of representatives, one from each distinct $\mu_m$-orbit of roots of $f(t)$. By \cite[Definition 7.23, Remark 7.24]{GenOort}, the solution $f(t)$ is isolated if the Vandermonde-like matrix $(x_j^{q-1})_{q \in S, 1 \leq j \leq N/m}$ is invertible, where $S$ is the set
\begin{equation*}
S := \{1 \leq q \leq N+\tilde{u}-1 \mid q \equiv -1 \pmod{m}, p \nmid q\}.
\end{equation*}

Now let $p$ be an odd prime, and $G \cong \ints/p^n \rtimes \ints/m$ be non-abelian with $p \nmid m$. When $n=1$, it was shown in \cite{BW:ll} and \cite{BWZ:dd} that Conjecture \ref{Cmain} holds. For the rest of this paper we assume $n \geq 2$. By \cite[Remark 1.16]{GenOort}, Conjecture \ref{Cmain} is reduced to showing that the isolated differential data criterion holds for certain finitely many quadruples $(p,m,\tilde{u},N)$. Using the same argument the following result shows that it is sufficient to check even fewer quadruples.

\begin{proposition}\label{prop_best_bound}
Let $p$ be an odd prime, and $m$ be a positive integer dividing $p-1$. Let $n \geq 2$. Suppose that for each $m-1 \leq \tilde{u} \leq (p^{n-2}+\cdots+1)(mp-1)$ with $p^{n-1} \nmid \tilde{u}$ and $\tilde{u} \equiv -1 \pmod{m}$, the isolated differential data criterion holds for the quadruples $(p,m,\tilde{u},(p-1)\tilde{u})$ and $(p,m,\tilde{u},(p-1)\tilde{u}-m)$. Then Conjecture~\ref{Cmain} holds for all non-abelian groups  $\ints/p^n \rtimes \ints/m$.
\end{proposition}

\begin{proof}
By Remark~\ref{Rcenterfree}, Conjecture~\ref{Cmain} is vacuous unless $\ints/p^n \rtimes \ints/m$ is center-free. So we may assume this is true. In particular, $m|(p-1)$. By \cite[Proposition 1.11]{GenOort}, it suffices to prove Conjecture~\ref{Cmain} for $\ints/p^n \rtimes \ints/m$-extensions whose $\ints/p^n$-subextension has  upper ramification breaks $(u_1,\ldots, u_n)$ with $u_1 < mp$ and $pu_{i-1} \leq u_i < pu_{i-1} + mp$ for all $1 < i \leq n$.  By \cite[Lemma 19]{Pries}, we note that $p^{n-1} \nmid u_i$ for any $1 \leq i < n$. Let $L/k[[s]]$ be such an extension. Set $u_0 = 0$. By \cite[Proposition 1.14]{GenOort}, if for each $1 < i \leq n$, both $u_i \equiv -1 \pmod{m}$ and $pu_{i-1} \leq u_i \leq pu_{i-1}+mp-1$, and also the isolated differential data criterion holds for $(p,m,u_{i-1},N)$ where 
\begin{equation*}
N = \begin{cases}
    (p-1)u_{i-1}, & \text{if } u_i = pu_{i-1} \text{ and}\\
    (p-1)u_{i-1}-m,               & \text{otherwise},
\end{cases}
\end{equation*}
then the extension $L/k[[s]]$ lifts to characteristic $0$. This inductive criterion on $u_i$ means $u_1 \leq mp-1$ and for $i \geq 2$, we have $u_{i-1} \leq (p^{i-2}+\cdots+1)(mp-1)$. By our hypothesis, these conditions are satisfied and the extension $L/k[[s]]$ lifts.
\end{proof}

\begin{remark}\label{rmk_best_bound}
Note that the upper bound on $\tilde{u}$ in Proposition \ref{prop_best_bound} is stronger than the bound $\tilde{u} \leq m(p^{n-1}+\cdots+p)$ listed in \cite[Remark 1.16]{GenOort}. 
\end{remark}

In the next sections we reduce the isolated differential data criterion for each quadruple into an equivalent statement about the existence of solutions to a system of multivariate polynomial equations, which is much easier to implement computationally.

\section{The multinomial coefficient approach}\label{Smultinomial}
Throughout this section let $p$ be an odd prime, and $k$ be an algebraically closed field of characteristic $p$. Let $(p,m,\tilde{u},N)$ be a quadruple satisfying the condition of Definition \ref{def_ddc}. Our first result shows that the differential data criterion can be formulated in terms of the existence of a solution to a system of polynomial equations with coefficients in $k$, hence can be studied using computational techniques. For any polynomial $f(t) \in k[t^m]$ of degree $N$ with $m|N$, we write
\begin{equation}\label{f-exp}
f=f(t)=\sum_{i=0}^{N/m} a_{m i} t^{mi}\in k[t^m].
\end{equation}

\begin{proposition}\label{system}
A polynomial $f(t) = \sum_{i=0}^{N/m} a_{mi}t^{mi}$ is a solution to the differential data criterion for $(p,m,\tilde{u},N)$ if and only if there are elements $a_{mi} \in k$, $0 \leq i \leq N/m$, satisfying the following system of equations.
\begin{equation}\label{d.d.c.6}
\begin{cases}
a_0=-u^{-1}\\
ua_{m i}^p = c_{p m i-\tilde{u}(p-1)}, \text{ for } \lceil \tilde{u}-\tilde{u}/p\rceil \leq m i \leq N\\
a_{m i}=0, \text{ otherwise} \\
a_N \neq 0,  \hspace{5mm}   \\  
\end{cases}
\end{equation}
where for $0\leq j \leq (p-1)N/m$, the $c_{mj}$'s are given by
\begin{equation}\label{eqncj}
c_{mj}=
\sum_{\substack{(s_0,\ldots,s_{N/m})\in\mathbb{N}^{N/m+1} \\s_0+s_1+\cdots+s_{N/m} =p-1\\ ms_1+ 2ms_2 +\cdots + Ns_{N/m}=mj}}
\binom{p-1}{s_0,\ldots,s_{N/m}}
\prod_{i=0}^{N/m}a_{m i}^{s_{i}} \in k \left[\{a_{m i}\}_{0\leq i\leq N/m}\right].
\end{equation}
\end{proposition}

\begin{proof}
Note that the coefficient of $t^{mj}$ in the expansion of the $(p-1)$-th power of $f(t)$ is $c_{mj}$ given by Equation (\ref{eqncj}) and so
\begin{equation*}\label{fpower-exp}
f^{p-1}=\sum_{j=0}^{(p-1)N/m} c_{m j} t^{m j}\in k[t^m].
\end{equation*}
By Definition \ref{def_ddc}, $(p,m,\tilde{u},N)$ satisfies the differential data criterion if and only if there exists a polynomial $f\in k[t^m]$ such that
\begin{equation*}\label{d.d.c.1}
\mathcal{C}\left( \frac{1}{t^{\tilde{u}+1}f}dt\right)=\frac{1+uf}{t^{\tilde{u}+1}f}dt.
\end{equation*}
This is equivalent to
\begin{equation*}\label{d.d.c.2}
\mathcal{C}\left( t^{(p-1)(\tilde{u}+1)}f^{p-1}dt\right)=(1+uf)dt.
\end{equation*}
Substituting the explicit forms of $f$ and $f^{p-1}$ we have
\begin{equation*}\label{d.d.c.4}
\mathcal{C}\left( \displaystyle\sum_{j=0}^{(p-1)N/m} c_{m j} t^{m j+\tilde{u}(p-1)}t^{p-1}dt\right) = \left(1+u\displaystyle\sum_{i=0}^{N/m} a_{m i} t^{m i}\right)dt.
\end{equation*}
From the basic properties of the Cartier operator we know that $\mathcal{C}$ is an additive map that sends any differential form $g^p t^{p-1} dt$ to $gdt$, and which kills any term $ct^b dt$ where $b \not \equiv -1 \pmod{p}$. So the above equation is equivalent to
\begin{equation*}\label{d.d.c.5}
\sum_{\substack{0 \leq j \leq (p-1)N/m \\ p| (mj+\tilde{u}(p-1))}} c_{m j} t^{m j+\tilde{u}(p-1)} = 
\left(1 + u\displaystyle\sum_{i=0}^{N/m} a_{m i} t^{m i}\right)^p = 1 + u\displaystyle\sum_{i=0}^{N/m} a_{m i}^p t^{p m i}.
\end{equation*}
Since the least exponent of the left hand side is $\tilde{u}(p-1)$, we obtain $1+ua_0^p=0$ and $a_{m i}=0$ when $p m i< \tilde{u}(p-1)$ or equivalently, when $m i <\tilde{u}-\tilde{u}/p$. For $m i \geq \lceil \tilde{u}-\tilde{u}/p\rceil$ we get $u a_{m i}^p=c_{m j}$ exactly when $m j+\tilde{u}(p-1)=p m i$, i.e. $u a_{m i}^p=c_{p m i-\tilde{u}(p-1)}$. The result follows.
\end{proof}

\begin{remark}\label{rmk_multi_coeff_homogen}
Note that each of the $c_{p m i-\tilde{u}(p-1)}$ is a homogeneous polynomial of degree $p-1$ in the variables $a_0$, $a_{\lceil \tilde{u} -\tilde{u}/p \rceil}$, $\cdots$, $a_N$.
\end{remark}

\begin{example}\label{ex32510differential}
By Proposition \ref{system}, the quadruple $(3,2,5,10)$ satisfies the differential data criterion if there exists a polynomial $f(t)=a_0+a_2t^2+a_4t^4+a_6t^6+a_8t^8+a_{10}t^{10}$ whose coefficients satisfy the following system of polynomial equations
\begin{equation*}
\begin{cases}
a_0 = -5^{-1} = 1, \\
5a_6^3 = -a_8, \\
5a_8^3 = -a_6a_8,  \\  
5a_{10}^3 = a_{10}^2, \\
a_2 = a_4 = 0, \\
a_{10} \neq 0.
\end{cases}
\end{equation*}
\noindent By a direct computation, one can show that the only solutions to this system are $f(t)=2t^{10}+1$ and $f(t)=2t^{10}+t^8+t^6+1$.
\end{example}

\section{Test for isolatedness}\label{Sisolated}

In this section, the notation $(g(i,j))_{i,j}$ means the matrix whose $ij^\text{th}$ entry is $g(i,j)$. Suppose $f(t)$ is a solution to the differential data criterion for a quadruple $(p, m, u, N)$.  As in Equation (\ref{f-exp}), write $f(t) = \sum_{i=0}^{N/m} a_{mi}t^{mi}$. By Lemma~\ref{Lfroots}, $f$ is separable and does not have $0$ as a root, so let $x_1, \ldots, x_{N/m}$ be a list consisting of one representative from each $\mu_m$-orbit of the roots of $f(t)$.  As was mentioned in \S\ref{Sdiff},
$f$ realizes the isolated 
differential data criterion if and only if the matrix
\[
M := (x_j^{q-1})_{q,j}
\]
is invertible, where $j$ ranges from $1$ to $N/m$ and $q$ ranges over all numbers from $1$ to $N + \tilde{u} - 1$ that are congruent to $-1 \pmod{m}$ and are not divisible by $p$.  In fact, since $\tilde{u} \equiv -1 \pmod{m}$ and $m \mid N$, the largest value of $q$ is $N + \tilde{u} - m$.  This matrix is always square (\cite[Remark~7.20]{GenOort}), and $x_{i_1}/x_{i_2} \notin \mu_m$ whenever $i_1 \neq i_2$.

In Corollary~\ref{Cisolated1} below, we give a criterion for 
the isolatedness in terms of the \emph{coefficients} of $f$, rather than its roots.  Indeed, one simply needs to check that a matrix made from coefficients of $f$ is invertible.  Since this criterion does not require computing roots of polynomials, it is computationally easier to verify.  

The starting point is a classical formula of Heineman (\cite{Heineman}) computing generalized Vandermonde determinants, where a \emph{generalized Vandermonde matrix} is a square matrix of the form $(z_j^{b_i})_{i,j}$ where the $a_i$ are integers, but $b_i$ are not necessarily equal to $i-1$.  The \emph{principal Vandermondian} associated to a generalized Vandermonde matrix is the determinant of the matrix given by $(z_j^{i-1})_{i,j}$, i.e., the standard Vandermonde determinant associated to the entries $z_1, \ldots, z_j$.

Since none of the $x_j$ is zero, we can form a new matrix $M'$ by dividing the $j^{\text{th}}$ column of $M$ by $x_j^{m-1}$, and $M'$ is invertible if and only if $M$ is. Now, if we let $y_j = x_j^m$, 
$M'$ can be expressed as follows:
\begin{equation}\label{Emprime}
M' = (y_j^{r})_{r,j},
\end{equation} where $1 \leq j \leq N/m$ as before and the $r$ range from 
$0$ to $(N + \tilde{u} - 2m + 1)/m$, skipping all values of $r$ such that $p \mid (mr
+ m - 1)$.  Since the $x_j$ lie in different $\mu_m$-orbits, the $y_j$ are 
pairwise distinct.  Thus $M'$ is a generalized Vandermonde matrix whose 
corresponding ``principal Vandermondian" (in the language of \cite{Heineman}) is nonzero.  
So it suffices to give a criterion for when the quotient of 
$\det(M')$ by this principal Vandermondian is $0$.

\begin{lemma}\label{Llargestindex}
Let $r_1, \ldots, r_{N/m}$ be the values of $r$ in the matrix $M'$ from (\ref{Emprime}) in ascending order.  Let $\epsilon$ be such that $N = (p-1)\tilde{u} - \epsilon m$ (so $\epsilon \in \{0,1\}$).
\begin{enumerate}[\upshape(i)]
\item We have $r_{N/m} - N/m + 1 = (\tilde{u}-m+1)/m$.  
\item  There are exactly $p-1-\epsilon$ values of $i$ such that $r_i - i + 1 = (\tilde{u}+m-1)/m$.

\end{enumerate}

\end{lemma}

\begin{proof}
For (i), it suffices to show that $r_{N/m} = (N+\tilde{u}-2m+1)/m$.  By the construction of $M'$, we need to only show that $p \nmid (N + \tilde{u} - m)$.  Since $N+\tilde{u} = p\tilde{u} - \epsilon m$, one needs only to show that $p$ does not divide $(\epsilon+1)m$.  This holds because $p \nmid m$ and $p > 2 \geq \epsilon+1$.

To prove (ii), note that (i) implies $mr_{N/m} + m - 1 = N + \tilde{u} - m = p\tilde{u} - (1+\epsilon)m$, so the largest value of $r$ less than $r_{N/m}$ such that $p \mid (mr+m-1)$ is that for which $mr + m - 1 = p\tilde{u} - pm$.  So $r_i - i + 1 = (\tilde{u}-m+1)/m$ exactly when $mr_i + m - 1 \geq p\tilde{u} - (p-1)m$.  Since $N= (p-1)\tilde{u} - \epsilon m$, this is exactly when $r_i \geq (N+\tilde{u}-(p-\epsilon)m + 1)/m$.  By (i), $r_{N/m} - (N+\tilde{u}-(p-\epsilon)m + 1)/m = p - 2 - \epsilon$.  We have proven the second assertion.
\end{proof}

\begin{remark}\label{Rjumps}
Note that, for each element $j$ of $\{1, 2, \ldots, (\tilde{u}-m+1)/m - 1\}$, there are exactly $p-1$ values of $i$ such that $r_i - (i-1) = j$.  Intuitively, we have $r_i - (i-1) = 0$ for $i = 1, 2, \ldots$ until we ``jump" over an $r$ such that $p \mid (mr + m - 1)$.  Then $r_i - (i-1) = 1$ for the next $p-1$ values of $i$, until we jump over another such $r$.  Then $r_i - (i-1) = 2$ for the next $p-1$ values of $i$, etc.  By Lemma~\ref{Llargestindex}, the largest value of $r_i - (i-1)$ occurs only $p-1-\epsilon$ times, as opposed to $p-1$ times.
\end{remark}

Echoing the notation of \cite{Heineman}, we write $V_i$ for the determinant of the matrix obtained by writing $(N/m + 1) \times N/m$ matrix
$$(y_j^r)_{r,j}$$
with $j$ ranging from $1$ to $N/m$ and $r$ ranging from $0$ to $N/m$, and then removing the $(N/m - i$)th row.  Note that $V_0$ is the principal Vandermondian, and is thus non-zero.  We have the following proposition.

\begin{proposition}[{\cite[Theorem I]{Heineman}}]\label{Psymfunction}
For all $i$, the ratio $V_i/V_0$ equals the $i$th elementary symmetric function in the $y_j$.
\end{proposition}
In particular, since the $y_j$ are the roots of the polynomial $\sum_{i=0}^{N/m} a_{mi}t^{mi}$, we can write 
\begin{equation}\label{Ecoefficient}
V_i = (-1)^i a_{N - mi}/a_N.
\end{equation}

Following \cite{Heineman}, for all $\ell, n \in \mathbb{N}$, define $D_\ell^n$ to be the $\ell \times \ell$ matrix given by the upper-left hand corner of the infinite matrix below:
$$
\left(
\begin{array}{cccccccccc}
V_1 & V_2 & V_3 & \cdots & V_n & 0 & 0 & \cdots & \cdots & \cdots \\
V_0 & V_1 & V_2 & \cdots & \cdots & V_n & 0 & 0 & \cdots & \cdots \\
0 & V_0 & V_1 & \cdots & \cdots & \cdots & V_n & 0 & 0 & \cdots \\
0 & 0 & V_0 & V_1 & \cdots & \cdots & \cdots & V_n & 0 & \cdots \\
\vdots & \vdots & \vdots & \ddots & \ddots & \ddots & \ddots & \cdots & \cdots & \cdots \\

\end{array}
\right)
$$

For a sequence of non-negative integers $l \geq t_1 \geq t_2 \geq \cdots \geq t_s$, we define $D_{\ell}^n\{t_1,t_2,\ldots,t_s\}$ to be the $\ell \times \ell$ matrix formed as follows: 

\begin{enumerate}
\item Start with $D_{\ell}^n$. 
\item Increase the subscripts of the $V_i$ in rows $1$ through $t_s$ by $s$, in rows $t_s + 1$ through $t_{s-1}$ by $s-1$, in rows $t_{s-1}+1$ through $t_{s-2}$ by $s-2$, etc.  
\end{enumerate}
Here $V_i$ is defined to be $0$ whenever $i > n$, and a zero that precedes a $V_0$ should be replaced by $V_0$ whenever the subscripts in its row are increased by $1$ (effectively, each increase by 1 ``moves the row to the left").    
\begin{example}\label{EVmatrix}
The matrix $D_4^{18}\{4, 4, 4, 3, 3, 3, 3, 2, 2, 2, 2, 1, 1, 1, 1\}$ is equal to 
$$\left( \begin{array}{cccc}
V_{16} & V_{17} & V_{18} & 0 \\
V_{11} & V_{12} & V_{13} & V_{14} \\
V_{6} & V_{7} & V_{8} & V_{9} \\
V_{1} & V_{2} & V_{3} & V_{4} 
\end{array}\right).$$
\end{example}

We write $|n_1, n_2, \ldots, n_\ell|$ for the determinant of the generalized Vandermonde matrix given by $(y_j^{n_i})$.
The main result of \cite{Heineman} implies the following proposition.
\begin{proposition}[{cf.\ \cite[Theorem IV]{Heineman}}]\label{Pheinemanmain}
For any $n \geq s \in \mathbb{N}$, and any natural numbers $t_1 \geq t_2 \geq \cdots \geq t_s$, the generalized Vandermonde $n \times n$ determinant $|t_1, t_2, \ldots, t_s, n-s-1, n-s-2, \ldots, 1, 0|$ is, up to sign and multiples of the principal Vandermondian $V_0$, equal to the determinant of 
$$D_{t_1 - n + 1}^n\{t_2 - n +2, t_3 - n + 3, \ldots, t_s - n + s\}.$$ \end{proposition}

\begin{proposition}\label{PMtoB}
Let $r_1, r_2, \ldots, r_{N/m}$ be the values of $r$ in the matrix $M'$ from (\ref{Emprime}) in ascending order.  Assume $\tilde{u} + 1 > m$.  Then the generalized Vandermonde determinant $|r_n, r_{n-1}, \ldots, r_1|$ is, up to sign and multiples of the principal Vandermondian $V_0$, equal to the determinant of 
\begin{equation}\label{ED} D_{(\tilde{u}-m+1)/m}^{N/m}\left\{\underbrace{\frac{\tilde{u}-m+1}{m}, \ldots, \frac{\tilde{u}-m+1}{m}}_{p-2-\epsilon \text{ times}},\underbrace{\frac{\tilde{u}-2m+1}{m}, \ldots, \frac{\tilde{u}-2m+1}{m}}_{p-1 \text{ times}},\ldots, \underbrace{2, \ldots, 2}_{p-1 \text{ times}},\underbrace{1, \ldots, 1}_{p-1 \text{ times}}\right\}.
\end{equation}
\end{proposition}

\begin{proof}
By the construction of $M'$, the $r_i$
are the whole numbers from $0$ to $(N + \tilde{u} - m +1)/m$ in increasing order, skipping all $r$ such that $p \mid (mr + m - 1)$.  Let $n = N/m$.  In the language of Proposition~\ref{Pheinemanmain}, we have $r_{n+1-i} = t_i$, so $t_2 - n + 2 = r_{n-1} - (n-1) + 1$, $t_3 - n + 3 = r_{n-2} - (n-2) + 1$, and so forth.  By Lemma~\ref{Llargestindex}(ii) and Remark~\ref{Rjumps}, there are exactly $p-2-\epsilon$ values of $i$ (other than $i = n$) such that $r_i - i + 1 = (\tilde{u}+m-1)/m$, and $p-1$ values of $i$ such that $r_i - i + 1 = j$ for each $j$ between $1$ and $(\tilde{u}+m-1)/m - 1$. The proposition now follows from  Proposition~\ref{Pheinemanmain}, 
\end{proof}

\begin{corollary}[Isolatedness criterion]\label{Cisolated1}
Suppose $f = \sum_{i=0}^{N/m} a_{mi}t^{mi} \in k[t]$ is a solution to the differential data criterion for $(p, m, \tilde{u}, N)$. 
If $\tilde{u} + 1  = m$, then $f$ is automatically a solution to the isolated differential data criterion.

If $\tilde{u} + 1 > m$, then $f$ is a solution to the isolated differential data criterion if and only if the matrix $A$ is invertible, where $A$ is the square matrix of size $(\tilde{u}+1 - m)/m$ whose $ij$th entry is $$a_{(p-1)(m-1) - m(j-1) + pm(i-1)}.$$  Here, we set $a_i = 0$ for all $i < 0$ and $i > N$.  
\end{corollary}

\begin{proof}
Let $n = N/m$, and let $\epsilon$ be defined as in Lemma~\ref{Llargestindex}.
We have that $f$ is a solution to the isolated differential data criterion if and only if the matrix $M'$ from (\ref{Emprime}) is invertible.  In the language of generalized Vandermonde determinants, the determinant of $M'$ up to sign is $|r_n, r_{n - 1}, \ldots, r_1 = 0|,$ where the $r_i$ are as in Proposition~\ref{PMtoB}.
Note that, by Lemma~\ref{Llargestindex}(i), $r_n-n+1 = (\tilde{u}-m+1)/m$.  So if $\tilde{u} + 1 = m$, then $r_n = n-1$ and the matrix $M'$ is itself Vandermonde, and thus invertible.

Now consider $\tilde{u} + 1 > m$.  By Proposition~\ref{PMtoB}, $\det M'$ is, up to sign and multiples of the principal Vandermondian, equal to the determinant of $B$, where $B$ is the matrix from (\ref{ED}).
This means that we start with the matrix
$D^n_{(\tilde{u}-m+1)/m}$, move the indices in the first row up by 
$(p-1)(\tilde{u}-m+1)/m - 1 - \epsilon$, move the indices in the second row up by
$(p-1)(\tilde{u}-2m+1)/m - 1 - \epsilon$, move the indices in the third row up by
$(p-1)(\tilde{u}-3m+1)/m - 1 - \epsilon$, etc.  So the first row of $B$ begins with 
$V_{\alpha}$, where $\alpha = (p-1)(\tilde{u}-m+1)/m - \epsilon$, the second row 
begins with $V_{\alpha-p}$, the third begins with $V_{\alpha-2p}$, etc., and 
the indices on the respective $V_{(\cdot)}$ increase by $1$ as we move from 
left to right along any row.  We need to show that $B$ is invertible if and only if the matrix $A$ from the statement of the corollary is invertible. 

By (\ref{Ecoefficient}), $V_\alpha = (-1)^\alpha a_{N-m\alpha}/a_N$.  Since $N = (p-1)\tilde{u} 
- m\epsilon$, we have that for $\alpha = (p-1)(\tilde{u}-m+1)/m - \epsilon$, the 
entry $V_\alpha$ equals $(-1)^\alpha a_{(p-1)(m-1)}/a_N$.  So this is the top 
left entry of $B$.  Each step to the right increases the index of 
$V_{(\cdot)}$ by $1$, which, by (\ref{Ecoefficient}), decreases the 
corresponding index of $a_{(\cdot)}$ by $m$ and changes the sign.  Similarly, each step down decreases the index of $V_{(\cdot)}$ by $p$, which increases 
the corresponding index of $a_{(\cdot)}$ by $pm$ and changes the sign.  So 
the $ij$th entry of $B$ is $(-1)^{i+j+\alpha}(a_{(p-1)(m-1) - m(j-1) + 
pm(i-1)})/a_N$.  Multiplying the odd-numbered rows and columns of $B$ by $-1$, and then multiplying every entry by $(-1)^{\alpha}$, we obtain the matrix $B'$ whose $ij$th entry is $a_{(p-1)(m-1) - m(j-1) + 
pm(i-1)}/a_N$.  Clearly, $B$ is invertible if and only if $B'$ is.  Since $a_N B' = A$ and $a_N \neq 0$, we see that $B'$ is invertible if and only if $A$ is, and we are done.
\end{proof}

\begin{example}\label{Ematrixexample}
Suppose $(p,m,\tilde{u},N) = (5, 2, 9, 36)$.  A solution $f$ is isolated if the 
matrix $B := D_4^{18}\{4, 4, 4, 3, 3, 3, 3, 2, 2, 2, 2, 1, 1, 1, 1\}$ is invertible.  We saw in Example~\ref{EVmatrix} that
$$B = \left( \begin{array}{cccc}
V_{16} & V_{17} & V_{18} & 0 \\
V_{11} & V_{12} & V_{13} & V_{14} \\
V_{6} & V_{7} & V_{8} & V_{9} \\
V_{1} & V_{2} & V_{3} & V_{4} 
\end{array}\right).$$
By (\ref{Ecoefficient}), up to factors of $a_N = a_{36} \neq 0$, 
this matrix is equal to 
$$\left( \begin{array}{cccc}
a_4 & a_2 & a_0 & 0 \\
a_{14} & a_{12} & a_{10} & a_8 \\
a_{24} & a_{22} & a_{20} & a_{18} \\
a_{34} & a_{32} & a_{30} & a_{28},
\end{array}\right)$$
which is the matrix $A$ from Corollary~\ref{Cisolated1}.
\end{example}

\begin{remark}\label{Rindices}
Suppose $A$ is the matrix from Corollary~\ref{Cisolated1}.  Here are some observations about $\det A$ which could potentially be helpful for future computations concerning the isolated differential data criterion.
\begin{enumerate}[\upshape(i)]
\item Each term of $\det A$ is a monomial of degree $(\tilde{u}-m+1)/m$ in the $a_i$.  If the \emph{weight} of a monomial
$c\prod_{\ell = 1}^L a_{i_{\ell}}$ is defined to be $\sum_{\ell=1}^L i_\ell$, then the weight of each term of $\det A$ equals 
$$\sum_{\ell =1}^{(\tilde{u}-m+1)/m} \left((m-1)(p-1) + (\ell - 1)(p-1)m\right) = \frac{(p-1)(\tilde{u}-m+1)(\tilde{u}-1)}{2m}.$$  In Example~\ref{Ematrixexample}, the weight of each term is $64$.
\item The indices of the $a_i$ jump by $pm$ in every row and by $\tilde{u} - 2m + 1$ from the first entry in a row to the last one.  Thus if $\tilde{u} - 2m + 1 < pm$, no $a_i$ appears more than once in $A$. In order to prove Conjecture~\ref{Cmain} when the $p$-Sylow subgroup of $G$ has order $p^2$, one need only show the isolated differential data criterion holds for $(p, m, \tilde{u}, N)$ with $\tilde{u} < pm$ (Proposition~\ref{prop_best_bound}).  In particular, $\tilde{u} - 2m + 1 < pm$, so we may assume in this case that no term of $\det A$ has a repeated factor of $a_i$.  
\end{enumerate}

\end{remark}

\begin{example}
\label{ex32510isolated}
Recall from Example \ref{ex32510differential} that the only functions that satisfy differential data criterion are $2t^{10}+1$ and $2t^{10}+t^8+t^6+1$. Moreover, by Corollary \ref{Cisolated1}, the matrix $A$ associated to that quadruple is

$$A =  \left( \begin{array}{cc}
a_2 & a_0 \\
a_8 & a_6 \\
\end{array}\right) = \left( \begin{array}{cc}
0 & 1\\
a_8 & a_6 \\
\end{array}\right).$$

\noindent The determinant of the matrix is $-a_8$. Hence, only $2t^{10}+t^8+t^6+1$ verifies the isolated differential data criterion. 
\end{example}

\section{Gr{\"o}bner basis computation setup}\label{Sgrobner}
Recall that a quadruple $(p,m,\tilde{u},N)$ satisfies the isolated differential data criterion if there exist elements $a_i\in k,\;0\leq i\leq N$ that satisfy the conditions of Proposition \ref{system} and Corollary \ref{Cisolated1}. Using the notation of the former, we define an ideal $I=(g_0,g_m, g_{2m}, \ldots,g_N) \in k[a_0,\ldots,a_N]$ where
\begin{equation}
\label{eqngi}
\begin{cases}
g_0:=a_0+u^{-1}\\
g_{m  i}:=u a_{m i}^p- c_{p m i-\tilde{u}(p-1)},\text{ for } \lceil \tilde{u}-\tilde{u}/p\rceil \leq m i \leq N\\
g_{m  i}:=a_{m i}, \text{ otherwise}\\
\end{cases}
\end{equation}
while, in the notation of the latter, we write $h$ for the determinant of the matrix with entries $a_{(p-1)(m-1) - m(j-1) + pm(i-1)}$ from Corollary \ref{Cisolated1}, which is also a polynomial in $k[a_0,\ldots,a_N]$. 
\begin{remark}\label{algorithm-1}
Proposition \ref{system} and Corollary \ref{Cisolated1} can be combined to provide a method to verify whether a given quadruple $(p,m,\tilde{u},N)$ satisfies the isolated differential data criterion: first we find a $k$-rational point $(a_0,\ldots,a_N)$ in the affine variety $V(I)\subseteq \mathbb{A}^N$ with $a_N\neq 0$, then we compute the determinant $h(a_0,\ldots,a_N)$ and verify it is non-zero. This approach was illustrated in Example \ref{ex32510differential} and Example \ref{ex32510isolated}, where we verified that the quadruple $(3,2,5,10)$ satisfies the isolated differential data criterion.
\end{remark}
The obvious disadvantage of the above method is that it seems difficult to find an explicit formula for such a $k$-rational point for a general quadruple $(p,m,\tilde{u},N)$. Our alternative computational approach can be summarized as follows:
\begin{proposition}\label{grobner-criterion}
There exists a solution to the isolated differential data criterion for a quadruple $(p, m, \tilde{u}, N)$ if and only if the ideal
\[J=(g_0, g_m, g_{2m}, \ldots, g_N, 1-y a_N h)\subseteq k[a_0, a_m, a_{2m}, \ldots, a_N,y]\] is not the unit ideal.  This is equivalent to $1$ not being in the reduced Gr{\"o}bner basis of $J$ with respect to any term order.
\end{proposition}
\begin{proof}

By Proposition~\ref{system}, the quadruple $(p, m, \tilde{u}, N)$ satisfies the differential data criterion if and only if there exists a solution to the system (\ref{d.d.c.6}), i.e., if $g_0 = g_m = \cdots = g_N = 0$ and $a_N \neq 0$.  Furthermore, such a solution is isolated if and only if the matrix $A$ of  Corollary \ref{Cisolated1} is invertible, i.e. if and only if $h\neq 0$.  As $k$ is algebraically closed, it follows from Hilbert's Nullstellensatz that this happens if and only if neither $a_N$ nor $h$ lies in the radical of $(g_0, g_m, \ldots, g_N)$.  By \cite[\S 15, Corollary 35]{MR2286236}, this is equivalent to $(g_0, g_m, \ldots, g_N, 1 - ya_Nh)$ not being the unit ideal in $k[a_0, \ldots, a_N, y]$.
\end{proof}

\begin{example}
Consider the quadruple $(3,2,5,10)$ of Example \ref{ex32510differential} and Example \ref{ex32510isolated}. To apply Proposition \ref{grobner-criterion}, we consider the ideal
\[J=(a_0-1, 2a_6^3+a_8, 2a_8^3+a_6a_8, 2a_{10}^3-a_{10}^2, a_2, a_4, 1-ya_{10}a_8)\subseteq k[a_0,a_2,a_4,a_6,a_8,a_{10}],\]
then compute the Gr{\"o}bner Basis
\[\mathcal{G}=\left\{y^3 - a_{6}, a_{6}^2 - y, a_{6}a_{8} - y^2, a_{8}^2 - a_{6}, a_{6}y - a_{8}, a_{8}y - 1, a_{10} + 1\right\}\]
corresponding to the (degrevlex) monomial ordering $y>a_{10}> \ldots > a_0$, and finally verify that $1\notin\mathcal{G}$.
\end{example}

\begin{remark}
Ignoring the line $a_N \neq 0$, the system of equations (\ref{d.d.c.6}) is a system of $N/m + 1$ equations in the $N/m +1$ variables $a_0, a_m, a_{2m}, \ldots, a_N$ over the algebraically closed field $k$.  It is thus reasonable to expect a solution.  In fact, the solution space, if it exists, is always zero-dimensional.  This fact is not used in the sequel, so we only sketch the proof: If there were a positive-dimensional solution space, then the solution space of the corresponding homogenized system (say using a variable $x$) would non-trivially intersect the hyperplane at infinity given by $x=0$.  Since the $c_j$ all have degree $p-1$ in the $a_i$ by Remark~\ref{rmk_multi_coeff_homogen}, one sees that $x = 0$ would imply that $a_i = 0$ for all $i$, a contradiction.
\end{remark}

\section{Computational Results}\label{Sresults}

In this section we prove that $D_{25}$ and $D_{27}$ are local Oort groups.
\subsection{\texorpdfstring{$D_{25}$}{D25}}
 Since $D_{25}\cong \ints/5^2 \rtimes \ints/2$, we have that $p=5,\;m=2,\;n=2$ and so, by Proposition  \ref{prop_best_bound}, it suffices to verify that the isolated differential data criterion is satisfied for the quadruples $(5,2,\tilde{u},N)$ where $\tilde{u}< 10,\;\tilde{u}\equiv -1\mod 2,\;5\nmid \tilde{u}$ and $N\in\left\{4\tilde{u}, 4\tilde{u}-2  \right\}$. We thus have that $D_{25}$ is local Oort group if the isolated differential data criterion is satisfied for the quadruples $(5,2,\tilde{u},N)$ where
\begin{equation}\label{uN-D25}
(\tilde{u},N)\in \left\{(1,2), (1,4), (3,10), (3,12), (7,26), (7,28), (9,34), (9,36)\right\}.
\end{equation}

To verify the isolated differential data criterion for the above quadruples we use Remark \ref{algorithm-1}: for each pair $(\tilde{u},N)$ in (\ref{uN-D25}) we consider the system of equations
\begin{equation*}
\begin{cases}
a_0=-u^{-1}\\
u a_{2 i}^5 = c_{10 i-4\tilde{u}}, \text{ for } \lceil \tilde{u}-\tilde{u}/5\rceil \leq 2 i \leq N\\
a_{2 i}=0, \text{ otherwise},
\end{cases}
\end{equation*}
where the polynomials $c_{10 i-4\tilde{u}}$ are defined in (\ref{eqncj}). We note that in all cases $\tilde{u}$ is not divisible by $p$, so $\tilde{u}=u$. The small size of the input for this case allows us to explicitly solve the system.\footnote{The solutions are calculated by hand from the Gr\"{o}bner basis of the system, which is calculated using the program \texttt{sage\_GB\_calculation\_according\_to\_paper.txt}.
They can be checked using the program \texttt{sage\_poly\_checking\_validity\_and\_isolatedness.txt}.  Both programs are bundled with this paper on the arxiv (\texttt{arxiv:1912.12797}).} 
For each solution obtained, we compute the determinant of the square matrix of size $(\tilde{u}-1)/2$ whose $ij$th entry is $a_{4 - 2(j-1) + 10(i-1)}$, as in Corollary~\ref{Cisolated1}. In the table below we indicate one isolated solution per pair $(\tilde{u},N)$, noting that, in some cases, we have found more than one. (In the table below, $\alpha$ satisfies $\alpha^2=3$ in $\ol{\mathbb{F}}_5$.)

\begin{center}
\begin{tabular}{|l|l|}
  \hline
  $(\tilde{u},N)$ & Solution to the Isolated Differential Data Criterion\\ \hline \hline
  (1,2)& $t^2+4$  \\ \hline
  (1,4)& $t^4+4$ \\ \hline
  (3,10) & $2t^{10}+3t^8+t^4+3$ \\ \hline
  (3,12)& $2t^{12}+t^8+4t^4+3$ \\ \hline
  (7,26)& $2t^{26}+2t^{24}+4t^{20}+2t^{16}+t^{12}+t^8+2$ \\ \hline
  (7,28) & $3t^{28}+t^{26}+2t^{24}+3t^{16}+2t^{12}+2t^8+2$ \\ \hline
  (9,34) & $4t^{34}+t^{32}+3t^{30}+2t^{28}+4t^{26}+4t^{22}+t^{18}+3t^{16}+2t^{14}+3t^{12}+2t^{10}+1$ \\ \hline
  (9,36) &$4t^{36}+(\alpha-1)t^{32}+(3\alpha+1)t^{28}+2t^{24}+4t^{20}+3\alpha t^{16}+t^{12}+1$ \\ \hline
\end{tabular}
\end{center}

\subsection{\texorpdfstring{$D_{27}$}{D27}}
Since $D_{27}\cong \ints/3^3 \rtimes \ints/2$, we have that $p=3,\;m=2,\;n=3$ and so, by Proposition \ref{prop_best_bound}, it suffices to verify that the isolated differential data criterion is satisfied for the quadruples $(3,2,\tilde{u},N)$ where $\tilde{u} \leq 20,\;\tilde{u}\equiv -1\mod 2,\;9\nmid \tilde{u}$ and $N\in\left\{2\tilde{u}, 2\tilde{u}-2  \right\}$. We thus have that $D_{27}$ is local Oort group if the isolated differential data criterion is satisfied for the quadruples $(3,2,\tilde{u},N)$ where
\begin{multline}\label{uN-D27}
(\tilde{u},N)\in \{(3,4), (3,6), (5,8), (5,10), (7,12), (7,14),\\
(15,28), (15,30),(17,32),(17,34),(19,36),(19,38)\}.
\end{multline}

To verify the isolated differential data criterion for the above quadruples we use Remark \ref{algorithm-1}: for each pair $(\tilde{u},N)$ in (\ref{uN-D27}) we consider the system of equations
\begin{equation*}
\begin{cases}
a_0=-u^{-1}\\
u a_{2 i}^3 = c_{6 i-2\tilde{u}}, \text{ for } \lceil \tilde{u}-\tilde{u}/3\rceil \leq 2 i \leq N\\
a_{2 i}=0, \text{ otherwise},
\end{cases}
\end{equation*}
where the polynomials $c_{6 i-2\tilde{u}}$ are defined in (\ref{eqncj}). We note that when $\tilde{u}$ is divisible by $p$, we have $u = \tilde{u}/3$.
The small size of the input for this case allows us to explicitly compute solutions except when $(\tilde{u}, N) = (17, 34)$ or $(19, 36)$.\footnote{Again, we use the program \texttt{sage\_GB\_calculation\_according\_to\_paper.txt} to compute the Gr\"{o}bner bases and check using \texttt{sage\_poly\_checking\_validity\_and\_isolatedness.txt}.} 
  In all but these two cases, we compute the determinant of the square matrix of size $(\tilde{u}-1)/2$ whose $ij$th entry is $a_{2 - 2(j-1) + 6(i-1)}$, as in Corollary~\ref{Cisolated1}. As with $D_{25}$, we indicate one solution to the isolated differential criterion per pair in the table below (in which $\beta$ satisfies $\beta^2=2$).

For the cases of $(17,34)$ and $(19,36)$, we instead use the criterion of Proposition \ref{grobner-criterion} to verify that an isolated solution exists.\footnote{This is verfied by the program \texttt{sage\_existence\_iso\_sol\_check.txt}, also available at \texttt{arxiv:1912.12797}.}
\\

\begin{center}
\begin{tabular}{|l|l|}
  \hline
  $(\tilde{u},N)$ & Solution to the Isolated Differential Data Criterion\\ \hline \hline
(3,4)& $\beta t^4+t^2+2$  \\ \hline
(3,6) & $t^6+2t^4+t^2+2$ \\ \hline
(5,8) & $t^8+t^6+1$ \\ \hline
(5,10)& $2t^{10}+t^8+t^6+1$ \\ \hline
(7,12)& $\beta t^{12}+t^{10}+t^8+2$ \\ \hline
(7,14)& $t^{14}+t^{12}+t^{10}+t^8+2$ \\ \hline
(15,28)& $\beta t^{28}+2t^{26}+2t^{24}+2t^{18}+2t^{16}+2t^{10}+1$ \\ \hline
(15,30)&$2t^{30}+2t^{28}+t^{24}+t^{20}+t^{18}+t^{16}+2t^{10}+1$ \\ \hline
(17,32) &$t^{32}+t^{30}+t^{28}+t^{26}+t^{24}+t^{22}+t^{20}+t^{18}+1$ \\ \hline
(17,34)& Solution exists due to Proposition \ref{grobner-criterion} \\ \hline
(19,36)& Solution exists due to Proposition \ref{grobner-criterion} \\ \hline
(19,38) &$t^{38}+t^{36}+t^{34}+t^{32}+t^{30}+t^{28}+t^{26}+t^{24}+t^{22}+t^{20}+2$ \\ \hline
\end{tabular}
\end{center}

\bibliographystyle{alpha}
\bibliography{mybib}

\end{document}